\documentclass{kms-c}

\issueinfo{}% volume number
  {}%        % issue number
  {}%        % month
  {}%     % year
\pagespan{1}{}
\copyrightinfo{}%              % copyright year
  {Korean Mathematical Society}% copyright holder
\usepackage{graphicx}
\allowdisplaybreaks

\usepackage{amsfonts}
\usepackage{amsmath}

\theoremstyle{plain}
\newtheorem{theorem}{Theorem}[section]

\newtheorem{corollary}[theorem]{Corollary}

\theoremstyle{definition}

\theoremstyle{remark}

\begin{document}

\title[On higher order Horadam 3-parameter generalized quaternions]
{On higher order Horadam 3-parameter generalized quaternions}

\author[G. Morales]{Gamaliel Morales}
\address{Gamaliel Morales \\ Instituto de Matem\'aticas \\ Pontificia Universidad Cat\'olica de Valpara\'iso \\ Blanco Viel 596, Chile}
\email{gamaliel.cerda.m@mail.pucv.cl}

%\thanks{This work was financially supported by KRF 2003-041-C20009}


\begin{abstract}
Recently, Kulo\u{g}lu {\it et al.} \cite{Kul} introduced the higher order Horadam numbers. In this study, novel 3-parameter generalized quaternion sequences of higher order Horadam numbers, which have not been studied before, are defined by investigating the relationship between generalized quaternions, which are important mathematical objects used in physics and mathematics, and higher order Horadam quaternions, which are extensions of the higher order Horadam numbers to quaternion algebra. Also, the recurrence relations of sequences whose members are higher order Horadam 3-parameter generalized quaternions are described. Furthermore, certain properties of these generalized quaternions are presented, such as the generating and exponential function, summation and Binet formula, and some identities resulting from these quaternions are obtained.

\end{abstract}

\maketitle

\section{Introduction}
Kulo\u{g}lu {\it et al.} in \cite{Kul} introduced a number sequence $\left(W_{n}(s)\right)_{n\geq 0}$, called higher order Horadam sequence and $s\geq 1$ is a fixed integer. This sequence is defined as follows
\begin{equation}\label{e1}
W_{n}(s)=\frac{W_{sn}}{W_{s}}=\frac{1}{A\alpha^{s}-B\beta^{s}}\left[A\alpha^{sn}-B\beta^{sn}\right],
\end{equation}
where $W_{n}=\frac{1}{\alpha-\beta}\left[A\alpha^{n}-B\beta^{n}\right]$ is the $n$th Horadam number (also $\alpha$ and $\beta$ are distinct and nonzero real numbers, $A=W_{1}-W_{0}\beta$ and $B=W_{1}-W_{0}\alpha$). Moreover, the authors proved that the recurrence relation for the higher order Horadam numbers has the following form
\begin{equation}\label{e2}
W_{n+2}(s)=\left(\alpha^{s}+\beta^{s}\right)W_{n+1}(s)-\left(\alpha\beta\right)^{s}W_{n}(s),\ \ n\geq 0,
\end{equation}
with initial conditions $W_{0}(s)=\frac{1}{A\alpha^{s}-B\beta^{s}}\left[W_{0}(\alpha-\beta)\right]$ and $W_{1}(s)=1$. The roots of the characteristic equation $\lambda^{2}-\left(\alpha^{s}+\beta^{s}\right)\lambda+\left(\alpha\beta\right)^{s}=0$, associated with the recurrence relation (\ref{e2}) are $\alpha^{s}$ and $\beta^{s}$. Considering the initial values $W_{0}=0$ and $W_{1}=1$ in the Horadam sequence, we obtain a particular case of the higher order Horadam numbers, called higher order generalized Fibonacci numbers. This sequence is denoted by
\begin{equation}\label{e3}
U_{n}(s)=\frac{U_{sn}}{U_{s}}=\frac{1}{\alpha^{s}-\beta^{s}}\left[\alpha^{sn}-\beta^{sn}\right],
\end{equation}
where $\alpha$, $\beta$ are given by (\ref{e1}). The particular case $\alpha+\beta=1$ and $\alpha\beta=-1$, is called higher order Fibonacci numbers and denoted by $F_{n}(s)$. Many interesting properties of the higher order Horadam numbers are given in \cite{Kul}. Among others, the well-known are
\[
\left[W_{n}(s)\right]^{2}-W_{n-1}(s)W_{n+1}(s)=AB(\alpha\beta)^{s(n-1)}\left[\frac{\alpha^{s}-\beta^{s}}{A\alpha^{s}-B\beta^{s}}\right]^{2} \ (n\geq 1),
\]
\[
\left[W_{n}(s)\right]^{2}-W_{n-m}(s)W_{n+m}(s)=AB(\alpha\beta)^{s(n-m)}\left[\frac{\alpha^{sm}-\beta^{sm}}{A\alpha^{s}-B\beta^{s}}\right]^{2} 
\]
and
\[
W_{m+1}(s)W_{n}(s)-W_{m}(s)W_{n+1}(s)=AB(\alpha\beta)^{sm}\frac{\left[\alpha^{s}-\beta^{s}\right]\left[\alpha^{s(n-m)}-\beta^{s(n-m)}\right]}{\left[A\alpha^{s}-B\beta^{s}\right]^{2}},
\]
where $n\geq m$. The last identity has a different presentation than the one given in \cite{Kul}. Also, note that $AB=W_{1}^{2}-(\alpha+\beta)W_{0}W_{1}+\alpha\beta W_{0}^{2}$ in all three identities above. Similarly, we obtain
\[
\left[U_{n}(s)\right]^{2}-U_{n-1}(s)U_{n+1}(s)=(\alpha\beta)^{s(n-1)} \ (n\geq 1),
\]
\[
\left[U_{n}(s)\right]^{2}-U_{n-m}(s)U_{n+m}(s)=(\alpha\beta)^{s(n-m)}\left[U_{m}(s)\right]^{2} \ (n\geq m)
\]
and
\[
U_{m+1}(s)U_{n}(s)-U_{m}(s)U_{n+1}(s)=(\alpha\beta)^{sm}U_{n-m}(s).
\]
Note that if we replace $s=1$ in equations (\ref{e1}) and (\ref{e3}), we obtain the relations $W_{n}(1)=\frac{1}{A\alpha-B\beta}\left[A\alpha^{n}-B\beta^{n}\right]$ and $U_{n}(1)=U_{n}=\frac{1}{\alpha-\beta}\left[\alpha^{n}-\beta^{n}\right]$.

In this paper, we will use the following identities of the higher order Horadam numbers:
\begin{equation}\label{id1}
\sum_{r=0}^{n}W_{r}(s)=\frac{(\alpha\beta)^{s}W_{n}(s)-W_{n+1}(s)+W_{1}(s)+\left(1-(\alpha^{s}+\beta^{s})\right)W_{0}(s)}{(\alpha\beta)^{s}-(\alpha^{s}+\beta^{s})+1},
\end{equation}
\begin{equation}\label{id2}
\sum_{n=0}^{\infty}W_{n}(s)\lambda^{n}=\frac{W_{0}(s)+\left[W_{1}(s)-(\alpha^{s}+\beta^{s})W_{0}(s)\right]\lambda}{1-(\alpha^{s}+\beta^{s})\lambda+(\alpha\beta)^{s}\lambda^{2}}
\end{equation}
and
\begin{equation}\label{id3}
W_{n}(s)=W_{0}(s)U_{n+1}(s)+\left[1-(\alpha^{s}+\beta^{s})W_{0}(s)\right]U_{n}(s).
\end{equation}

\section{On higher order Horadam 3-parameter generalized quaternions}
On the other hand, \c{S}ent\"urk and \"Unal \cite{Se} investigated a new quite comprehensive quaternion type called 3-parameter generalized quaternions (shortly 3PGQs). The authors constructed a new and general aspect for the quaternion algebra depending on the 3-parameters to get a generalization of real, hyperbolic and 2-parameter generalized quaternions. A 3-parameter generalized quaternion is denoted by
\[
\textbf{Q}=x_{0}\textbf{1}+x_{1}\textbf{i}+x_{2}\textbf{j}+x_{3}\textbf{k},\ \ x_{r}\in \mathbb{R}\ (r=0,1,2,3),
\]
where quaternionic units $\{\textbf{1}, \textbf{i}, \textbf{j}, \textbf{k}\}$ hold the rules given in Table \ref{t1}. Taking $\lambda_{1}=\lambda_{2}=\lambda_{3}=1$ in Table \ref{t1}, we obtain the classical Hamilton quaternions with real coefficients. The conjugate of a 3-parameter generalized quaternion is given by $\textbf{Q}^{\dagger}=x_{0}\textbf{1}-x_{1}\textbf{i}-x_{2}\textbf{j}-x_{3}\textbf{k}$, the norm is 
\[
\textbf{Q}\textbf{Q}^{\dagger}=x_{0}^{2}+\lambda_{1}\lambda_{2}x_{1}^{2}+\lambda_{1}\lambda_{3}x_{2}^{2}+\lambda_{2}\lambda_{3}x_{3}^{2}.
\]
For the basics on 3-parameter generalized quaternions theory, see \cite{Se}. The 3-parameter generalized quaternion of sequences firstly were considered in 2024 by K\i z\i late\c{s} and Kibar \cite{Ki3}. They introduced higher order Fibonacci 3-parameter generalized quaternions by the equations
\[
\textbf{Q}F_{n}(s)=F_{n}(s)\textbf{1}+F_{n+1}(s)\textbf{i}+F_{n+2}(s)\textbf{j}+F_{n+3}(s)\textbf{k},
\]
where $F_{n}(s)$ denotes the $n$th higher order Fibonacci number.
\begin{table}
		\centering
\begin{tabular}{|c|c|c|c|c|}
\hline
$\downarrow \times \rightarrow$ & $\textbf{1}$ & $\textbf{i}$ & $\textbf{j}$ & $\textbf{k}$ \\
\hline
$\textbf{1}$ & $\textbf{1}$ & $\textbf{i}$ & $\textbf{j}$ & $\textbf{k}$ \\
\hline
$\textbf{i}$ & $\textbf{i}$ & $-\lambda_{1}\lambda_{2}\textbf{1}$ & $\lambda_{1}\textbf{k}$ & $-\lambda_{2}\textbf{j}$ \\
\hline
$\textbf{j}$ & $\textbf{j}$ & $-\lambda_{1}\textbf{k}$ & $-\lambda_{1}\lambda_{3}\textbf{1}$ & $\lambda_{3}\textbf{i}$ \\
\hline
$\textbf{k}$ & $\textbf{k}$ & $\lambda_{2}\textbf{j}$ & $-\lambda_{3}\textbf{i}$ & $-\lambda_{2}\lambda_{3}\textbf{1}$ \\
\hline
\end{tabular}
			\caption{3-parameter generalized quaternion multiplication table.}
			\label{t1} 
\end{table} 

The hyper complex numbers of the well-known higher order sequences have been investigated by several authors. For example, in \cite{Ki1} the higher order Fibonacci quaternions were introduced, in \cite{Ki2} the higher order Fibonacci hyper complex numbers were considered. In \cite{Ay,Ha,Ku,Ozi1,Ozi2,Ozk,Po,Uy} many interesting identities of some particular cases of integer sequences and hyper complex numbers  were established.

Here, we introduce higher order Horadam 3-parameter generalized quaternions and derive some identities such as Binet formulas, generating function, Catalan identities, d'Ocagne identities for these new type of generalized quaternions.

Let $n\geq 0$. The higher order Horadam 3-parameter generalized quaternion sequence $\left(\textbf{Q}W_{n}(s)\right)$ we define by the following relation
\begin{equation}\label{f1}
\textbf{Q}W_{n}(s)=W_{n}(s)\textbf{1}+W_{n+1}(s)\textbf{i}+W_{n+2}(s)\textbf{j}+W_{n+3}(s)\textbf{k},
\end{equation}
where $W_{n}(s)$ denotes the $n$th higher order Horadam number. In the same way we can define the higher order generalized Fibonacci 3-parameter generalized quaternion sequence  $\left(\textbf{Q}U_{n}(s)\right)$ as
\begin{equation}\label{f2}
\textbf{Q}U_{n}(s)=U_{n}(s)\textbf{1}+U_{n+1}(s)\textbf{i}+U_{n+2}(s)\textbf{j}+U_{n+3}(s)\textbf{k},
\end{equation}
where $U_{n}(s)$ is defined by (\ref{e3}).

The next theorems present some basic properties of the higher order Horadam 3-parameter generalized quaternions.
\begin{theorem}\label{t1}
Let $n\geq 0$ be an integer. Then
\[
\textnormal{\textbf{Q}}W_{n+2}(s)=(\alpha^{s}+\beta^{s})\textnormal{\textbf{Q}}W_{n+1}(s)-(\alpha\beta)^{s}\textnormal{\textbf{Q}}W_{n}(s),
\]
where $\textnormal{\textbf{Q}}W_{0}(s)$ and $\textnormal{\textbf{Q}}W_{1}(s)$ are given in (\ref{f1}).
\end{theorem}
\begin{proof}
By formulas (\ref{f1}) and (\ref{e2}), we get
\begin{align*}
(\alpha^{s}+\beta^{s})&\textnormal{\textbf{Q}}W_{n+1}(s)-(\alpha\beta)^{s}\textnormal{\textbf{Q}}W_{n}(s)\\
&=(\alpha^{s}+\beta^{s})\left[W_{n+1}(s)\textbf{1}+W_{n+2}(s)\textbf{i}+W_{n+3}(s)\textbf{j}+W_{n+4}(s)\textbf{k}\right]\\
&\ \ - (\alpha\beta)^{s}\left[W_{n}(s)\textbf{1}+W_{n+1}(s)\textbf{i}+W_{n+2}(s)\textbf{j}+W_{n+3}(s)\textbf{k}\right]\\
&=\left[\left(\alpha^{s}+\beta^{s}\right)W_{n+1}(s)-\left(\alpha\beta\right)^{s}W_{n}(s)\right]\textbf{1}\\
&\ \ + \left[\left(\alpha^{s}+\beta^{s}\right)W_{n+2}(s)-\left(\alpha\beta\right)^{s}W_{n+1}(s)\right]\textbf{i}\\
&\ \ + \left[\left(\alpha^{s}+\beta^{s}\right)W_{n+3}(s)-\left(\alpha\beta\right)^{s}W_{n+2}(s)\right]\textbf{j}\\
&\ \ + \left[\left(\alpha^{s}+\beta^{s}\right)W_{n+4}(s)-\left(\alpha\beta\right)^{s}W_{n+3}(s)\right]\textbf{k}\\
&=W_{n+2}(s)\textbf{1}+W_{n+3}(s)\textbf{i}+W_{n+4}(s)\textbf{j}+W_{n+5}(s)\textbf{k}\\
&=\textnormal{\textbf{Q}}W_{n+2}(s),
\end{align*}
which ends the proof.
\end{proof}

\begin{theorem}\label{t2}
Let $n\geq 0$ be an integer. Then
\[
\textnormal{\textbf{Q}}W_{n}(s)\textnormal{\textbf{Q}}^{\dagger}W_{n}(s)=\frac{A^{2}\alpha^{2sn}\Phi_{\alpha}(s)-2AB(\alpha\beta)^{sn}\Phi_{\alpha\beta}(s)+B^{2}\beta^{2sn}\Phi_{\beta}(s)}{\left[A\alpha^{s}-B\beta^{s}\right]^{2}},
\]
where $\Phi_{\alpha}(s)=1+\lambda_{1}\lambda_{2}\alpha^{2s}+\lambda_{1}\lambda_{3}\alpha^{4s}+\lambda_{2}\lambda_{3}\alpha^{6s}$, $\Phi_{\alpha\beta}(s)=1+\lambda_{1}\lambda_{2}(\alpha\beta)^{s}+\lambda_{1}\lambda_{3}(\alpha\beta)^{2s}+\lambda_{2}\lambda_{3}(\alpha\beta)^{3s}$ and $\Phi_{\beta}(s)=1+\lambda_{1}\lambda_{2}\beta^{2s}+\lambda_{1}\lambda_{3}\beta^{4s}+\lambda_{2}\lambda_{3}\beta^{6s}$.
\end{theorem}
\begin{proof}
Using the definition of the conjugate of a 3-parameter generalized quaternion, we obtain 
\begin{align*}
&\left[A\alpha^{s}-B\beta^{s}\right]^{2}\textnormal{\textbf{Q}}W_{n}(s)\textnormal{\textbf{Q}}^{\dagger}W_{n}(s)\\
&=\left[A\alpha^{s}-B\beta^{s}\right]^{2}\left\lbrace \begin{array}{c}  \left[W_{n}(s)\right]^{2}+\lambda_{1}\lambda_{2}\left[W_{n+1}(s)\right]^{2}\\
+ \lambda_{1}\lambda_{3}\left[W_{n+2}(s)\right]^{2}+\lambda_{2}\lambda_{3}\left[W_{n+3}(s)\right]^{2} \end{array} 
\right\rbrace \\
&=A^{2}\alpha^{2sn}-2AB(\alpha\beta)^{sn}+B^{2}\beta^{2sn}\\
&\ \ + \lambda_{1}\lambda_{2}\left[A^{2}\alpha^{2s(n+1)}-2AB(\alpha\beta)^{s(n+1)}+B^{2}\beta^{2s(n+1)}\right]\\
&\ \ + \lambda_{1}\lambda_{3}\left[A^{2}\alpha^{2s(n+2)}-2AB(\alpha\beta)^{s(n+2)}+B^{2}\beta^{2s(n+2)}\right]\\
&\ \ + \lambda_{2}\lambda_{3}\left[A^{2}\alpha^{2s(n+3)}-2AB(\alpha\beta)^{s(n+3)}+B^{2}\beta^{2s(n+3)}\right]\\
&=A^{2}\alpha^{2sn}\left[1+\lambda_{1}\lambda_{2}\alpha^{2s}+\lambda_{1}\lambda_{3}\alpha^{4s}+\lambda_{2}\lambda_{3}\alpha^{6s}\right]\\\
&\ \ - 2AB(\alpha\beta)^{sn}\left[1+\lambda_{1}\lambda_{2}(\alpha\beta)^{s}+\lambda_{1}\lambda_{3}(\alpha\beta)^{2s}+\lambda_{2}\lambda_{3}(\alpha\beta)^{3s}\right]\\
&\ \ + B^{2}\beta^{2sn}\left[1+\lambda_{1}\lambda_{2}\beta^{2s}+\lambda_{1}\lambda_{3}\beta^{4s}+\lambda_{2}\lambda_{3}\beta^{6s}\right]\\
&=A^{2}\alpha^{2sn}\Phi_{\alpha}(s)-2AB(\alpha\beta)^{sn}\Phi_{\alpha\beta}(s)+B^{2}\beta^{2sn}\Phi_{\beta}(s),
\end{align*}
where $\Phi_{\alpha}(s)=1+\lambda_{1}\lambda_{2}\alpha^{2s}+\lambda_{1}\lambda_{3}\alpha^{4s}+\lambda_{2}\lambda_{3}\alpha^{6s}$, $\Phi_{\alpha\beta}(s)=1+\lambda_{1}\lambda_{2}(\alpha\beta)^{s}+\lambda_{1}\lambda_{3}(\alpha\beta)^{2s}+\lambda_{2}\lambda_{3}(\alpha\beta)^{3s}$ and $\Phi_{\beta}(s)=1+\lambda_{1}\lambda_{2}\beta^{2s}+\lambda_{1}\lambda_{3}\beta^{4s}+\lambda_{2}\lambda_{3}\beta^{6s}$. Hence we get the result.
\end{proof}

The next theorem gives the Binet formulas for the higher order Horadam 3-parameter generalized quaternions.
\begin{theorem}\label{t3}
Let $n\geq 0$ be an integer. Binet formulas for $\textnormal{\textbf{Q}}W_{n}(s)$ have the following form
\[
\textnormal{\textbf{Q}}W_{n}(s)=\frac{A\alpha^{sn}\Theta_{\alpha}(s)-B\beta^{sn}\Theta_{\beta}(s)}{A\alpha^{s}-B\beta^{s}},
\]
where $\Theta_{\alpha}(s)=\textnormal{\textbf{1}}+\alpha^{s}\textnormal{\textbf{i}}+\alpha^{2s}\textnormal{\textbf{j}}+\alpha^{3s}\textnormal{\textbf{k}}$ and $\Theta_{\beta}(s)=\textnormal{\textbf{1}}+\beta^{s}\textnormal{\textbf{i}}+\beta^{2s}\textnormal{\textbf{j}}+\beta^{3s}\textnormal{\textbf{k}}$.
\end{theorem}
\begin{proof}
By formulas (\ref{f1}) and (\ref{e1}), we get
\begin{align*}
\left[A\alpha^{s}-B\beta^{s}\right]&\textbf{Q}W_{n}(s)\\
&=\left[A\alpha^{s}-B\beta^{s}\right]\left[W_{n}(s)\textbf{1}+W_{n+1}(s)\textbf{i}+W_{n+2}(s)\textbf{j}+W_{n+3}(s)\textbf{k}\right]\\
&=\left[A\alpha^{sn}-B\beta^{sn}\right]\textbf{1}+\left[A\alpha^{s(n+1)}-B\beta^{s(n+1)}\right]\textbf{i}\\
&\ \ + \left[A\alpha^{s(n+2)}-B\beta^{s(n+2)}\right]\textbf{j}+\left[A\alpha^{s(n+3)}-B\beta^{s(n+3)}\right]\textbf{k}\\
&=A\alpha^{sn}\left[\textbf{1}+\alpha^{s}\textbf{i}+\alpha^{2s}\textbf{j}+\alpha^{3s}\textbf{k}\right]-B\beta^{sn}\left[\textbf{1}+\beta^{s}\textbf{i}+\beta^{2s}\textbf{j}+\beta^{3s}\textbf{k}\right]\\
&=A\alpha^{sn}\Theta_{\alpha}(s)-B\beta^{sn}\Theta_{\beta}(s),
\end{align*}
where $\Theta_{\alpha}(s)=\textbf{1}+\alpha^{s}\textbf{i}+\alpha^{2s}\textbf{j}+\alpha^{3s}\textbf{k}$ and $\Theta_{\beta}(s)=\textbf{1}+\beta^{s}\textbf{i}+\beta^{2s}\textbf{j}+\beta^{3s}\textbf{k}$. Hence we get the result.
\end{proof}

Considering the initial conditions $W_{0}=0$ and $W_{1}$ in the Theorem \ref{t3}, we obtain the following result for the higher order generalized Fibonacci 3-parameter generalized quaternions.
\begin{corollary}
Let $n\geq 0$ be an integer. Binet formulas for $\textnormal{\textbf{Q}}U_{n}(s)$ have the following form
\[
\textnormal{\textbf{Q}}U_{n}(s)=\frac{\alpha^{sn}\Theta_{\alpha}(s)-\beta^{sn}\Theta_{\beta}(s)}{\alpha^{s}-\beta^{s}},
\]
where $\Theta_{\alpha}(s)=\textnormal{\textbf{1}}+\alpha^{s}\textnormal{\textbf{i}}+\alpha^{2s}\textnormal{\textbf{j}}+\alpha^{3s}\textnormal{\textbf{k}}$ and $\Theta_{\beta}(s)=\textnormal{\textbf{1}}+\beta^{s}\textnormal{\textbf{i}}+\beta^{2s}\textnormal{\textbf{j}}+\beta^{3s}\textnormal{\textbf{k}}$.
\end{corollary}

\section{Catalan, Cassini and d'Ocagne identities for the higher order Horadam 3-parameter generalized quaternions}
Now we will give some identities for the higher order Horadam 3-parameter generalized quaternion, these identities are easily proved using the Binet formulas in Theorem \ref{t3}. Furthermore, using $\Theta_{\alpha}(s)=\textbf{1}+\alpha^{s}\textbf{i}+\alpha^{2s}\textbf{j}+\alpha^{3s}\textbf{k}$ and $\Theta_{\beta}(s)=\textbf{1}+\beta^{s}\textbf{i}+\beta^{2s}\textbf{j}+\beta^{3s}\textbf{k}$, we get
\[
\Theta_{\alpha}(s)\Theta_{\beta}(s)=\left\lbrace \begin{array}{c}
\left[1-\lambda_{1}\lambda_{2}(\alpha\beta)^{s}-\lambda_{1}\lambda_{3}(\alpha\beta)^{2s}-\lambda_{2}\lambda_{3}(\alpha\beta)^{3s}\right]\textbf{1}\\
+(\alpha^{s}+\beta^{s})\textbf{i}+(\alpha^{2s}+\beta^{2s})\textbf{j}+(\alpha^{s3}+\beta^{3s})\textbf{k}\\
-(\alpha\beta)^{s}(\alpha^{s}-\beta^{s})\left[\lambda_{3}(\alpha\beta)^{s}\textbf{i}-\lambda_{2}(\alpha^{s}+\beta^{s})\textbf{j}+\lambda_{1}\textbf{k}\right]
\end{array}
\right\rbrace 
\]
and 
\[
\Theta_{\beta}(s)\Theta_{\alpha}(s)=\left\lbrace \begin{array}{c}
\left[1-\lambda_{1}\lambda_{2}(\alpha\beta)^{s}-\lambda_{1}\lambda_{3}(\alpha\beta)^{2s}-\lambda_{2}\lambda_{3}(\alpha\beta)^{3s}\right]\textbf{1}\\
+(\alpha^{s}+\beta^{s})\textbf{i}+(\alpha^{2s}+\beta^{2s})\textbf{j}+(\alpha^{s3}+\beta^{3s})\textbf{k}\\
+(\alpha\beta)^{s}(\alpha^{s}-\beta^{s})\left[\lambda_{3}(\alpha\beta)^{s}\textbf{i}-\lambda_{2}(\alpha^{s}+\beta^{s})\textbf{j}+\lambda_{1}\textbf{k}\right]
\end{array}
\right\rbrace .
\]
The two-element commutator, $\Theta_{\beta}(s)$ and $\Theta_{\alpha}(s)$, is the element
\begin{equation}\label{c}
\begin{aligned}
\left[\Theta_{\beta}(s),\Theta_{\alpha}(s)\right]&=\Theta_{\beta}(s)\Theta_{\alpha}(s)-\Theta_{\alpha}(s)\Theta_{\beta}(s)\\
&=2(\alpha\beta)^{s}(\alpha^{s}-\beta^{s})\left[\lambda_{3}(\alpha\beta)^{s}\textbf{i}-\lambda_{2}(\alpha^{s}+\beta^{s})\textbf{j}+\lambda_{1}\textbf{k}\right].
\end{aligned}
\end{equation}

\begin{theorem}[Catalan identity for $\textnormal{\textbf{Q}}W_{n}(s)$]\label{t5}
Let $n\geq 0$, $m\geq 0$ be integers such that $n\geq m$. Then
\begin{align*}
\left[\textnormal{\textbf{Q}}W_{n}(s)\right]^{2}&-\textnormal{\textbf{Q}}W_{n-m}(s)\textnormal{\textbf{Q}}W_{n+m}(s)\\
&=\frac{AB(\alpha\beta)^{s(n-m)}\left(\alpha^{sm}-\beta^{sm}\right)}{\left[A\alpha^{s}-B\beta^{s}\right]^{2}}\left\lbrace \begin{array}{c} 
\left(\alpha^{sm}-\beta^{sm}\right)\Theta_{\alpha}(s)\Theta_{\beta}(s)\\
+ \alpha^{sm}\left[\Theta_{\beta}(s),\Theta_{\alpha}(s)\right]
\end{array}\right\rbrace ,
\end{align*}
where $\left[\Theta_{\beta}(s),\Theta_{\alpha}(s)\right]=2(\alpha\beta)^{s}(\alpha^{s}-\beta^{s})\left[\lambda_{3}(\alpha\beta)^{s}\textnormal{\textbf{i}}-\lambda_{2}(\alpha^{s}+\beta^{s})\textnormal{\textbf{j}}+\lambda_{1}\textnormal{\textbf{k}}\right]$.
\end{theorem}
\begin{proof}
By formula (\ref{c}) and Theorem \ref{t3}, we get
\begin{align*}
&\left[A\alpha^{s}-B\beta^{s}\right]^{2}\left\lbrace \left[\textnormal{\textbf{Q}}W_{n}(s)\right]^{2}-\textnormal{\textbf{Q}}W_{n-m}(s)\textnormal{\textbf{Q}}W_{n+m}(s)\right\rbrace \\
&=\left[A\alpha^{sn}\Theta_{\alpha}(s)-B\beta^{sn}\Theta_{\beta}(s)\right]\left[A\alpha^{sn}\Theta_{\alpha}(s)-B\beta^{sn}\Theta_{\beta}(s)\right]\\
&\ \ - \left[A\alpha^{s(n-m)}\Theta_{\alpha}(s)-B\beta^{s(n-m)}\Theta_{\beta}(s)\right]\left[A\alpha^{s(n+m)}\Theta_{\alpha}(s)-B\beta^{s(n+m)}\Theta_{\beta}(s)\right]\\
&=-AB(\alpha\beta)^{sn}\Theta_{\alpha}(s)\Theta_{\beta}(s)-AB(\alpha\beta)^{sn}\Theta_{\beta}(s)\Theta_{\alpha}(s)\\
&\ \ + AB(\alpha\beta)^{sn}\alpha^{-sm}\beta^{sm}\Theta_{\alpha}(s)\Theta_{\beta}(s)+AB(\alpha\beta)^{sn}\alpha^{sm}\beta^{-sm}\Theta_{\beta}(s)\Theta_{\alpha}(s)\\
&=AB(\alpha\beta)^{s(n-m)}\left(\alpha^{sm}-\beta^{sm}\right)\left[\alpha^{sm}\Theta_{\beta}(s)\Theta_{\alpha}(s)-\beta^{sm}\Theta_{\alpha}(s)\Theta_{\beta}(s)\right]\\
&=AB(\alpha\beta)^{s(n-m)}\left(\alpha^{sm}-\beta^{sm}\right)\left\lbrace \begin{array}{c} 
\left(\alpha^{sm}-\beta^{sm}\right)\Theta_{\alpha}(s)\Theta_{\beta}(s)\\
+ \alpha^{sm}\left[\Theta_{\beta}(s),\Theta_{\alpha}(s)\right]
\end{array}\right\rbrace ,
\end{align*}
where $\left[\Theta_{\beta}(s),\Theta_{\alpha}(s)\right]=2(\alpha\beta)^{s}(\alpha^{s}-\beta^{s})\left[\lambda_{3}(\alpha\beta)^{s}\textbf{i}-\lambda_{2}(\alpha^{s}+\beta^{s})\textbf{j}+\lambda_{1}\textbf{k}\right]$. Hence we get the result.
\end{proof}

Note that for $m=1$ in Theorem \ref{t5} we have the Cassini identity for the higher order Horadam 3-parameter generalized quaternions.
\begin{corollary}
Let $n$ be an integer such that $n\geq 1$. Then
\begin{align*}
&\left[\textnormal{\textbf{Q}}W_{n}(s)\right]^{2}-\textnormal{\textbf{Q}}W_{n-1}(s)\textnormal{\textbf{Q}}W_{n+1}(s)\\
&=\frac{AB(\alpha\beta)^{s(n-1)}\left(\alpha^{s}-\beta^{s}\right)}{\left[A\alpha^{s}-B\beta^{s}\right]^{2}}\left\lbrace \begin{array}{c} 
\left(\alpha^{s}-\beta^{s}\right)\Theta_{\alpha}(s)\Theta_{\beta}(s)+ \alpha^{s}\left[\Theta_{\beta}(s),\Theta_{\alpha}(s)\right]
\end{array}\right\rbrace ,
\end{align*}
where $\left[\Theta_{\beta}(s),\Theta_{\alpha}(s)\right]=2(\alpha\beta)^{s}(\alpha^{s}-\beta^{s})\left[\lambda_{3}(\alpha\beta)^{s}\textnormal{\textbf{i}}-\lambda_{2}(\alpha^{s}+\beta^{s})\textnormal{\textbf{j}}+\lambda_{1}\textnormal{\textbf{k}}\right]$.
\end{corollary}

\begin{theorem}[d'Ocagne identity for $\textnormal{\textbf{Q}}W_{n}(s)$]\label{t6}
Let $n\geq 0$, $m\geq 0$ be integers such that $n\geq m$. Then
\begin{align*}
\textnormal{\textbf{Q}}W_{m+1}(s)&\textnormal{\textbf{Q}}W_{n}(s)-\textnormal{\textbf{Q}}W_{m}(s)\textnormal{\textbf{Q}}W_{n+1}(s)\\
&=\frac{AB(\alpha\beta)^{sm}\left(\alpha^{s}-\beta^{s}\right)}{\left[A\alpha^{s}-B\beta^{s}\right]^{2}}\left\lbrace \begin{array}{c} 
\left(\alpha^{s(n-m)}-\beta^{s(n-m)}\right)\Theta_{\alpha}(s)\Theta_{\beta}(s)\\
+ \alpha^{s(n-m)}\left[\Theta_{\beta}(s),\Theta_{\alpha}(s)\right]
\end{array}\right\rbrace ,
\end{align*}
where $\left[\Theta_{\beta}(s),\Theta_{\alpha}(s)\right]=2(\alpha\beta)^{s}(\alpha^{s}-\beta^{s})\left[\lambda_{3}(\alpha\beta)^{s}\textnormal{\textbf{i}}-\lambda_{2}(\alpha^{s}+\beta^{s})\textnormal{\textbf{j}}+\lambda_{1}\textnormal{\textbf{k}}\right]$.
\end{theorem}
\begin{proof}
By formula (\ref{c}) and Theorem \ref{t3}, we get
\begin{align*}
&\left[A\alpha^{s}-B\beta^{s}\right]^{2}\left[\textnormal{\textbf{Q}}W_{m+1}(s)\textnormal{\textbf{Q}}W_{n}(s)-\textnormal{\textbf{Q}}W_{m}(s)\textnormal{\textbf{Q}}W_{n+1}(s)\right] \\
&=\left[A\alpha^{s(m+1)}\Theta_{\alpha}(s)-B\beta^{s(m+1)}\Theta_{\beta}(s)\right]\left[A\alpha^{sn}\Theta_{\alpha}(s)-B\beta^{sn}\Theta_{\beta}(s)\right]\\
&\ \ - \left[A\alpha^{sm}\Theta_{\alpha}(s)-B\beta^{sm}\Theta_{\beta}(s)\right]\left[A\alpha^{s(n+1)}\Theta_{\alpha}(s)-B\beta^{s(n+1)}\Theta_{\beta}(s)\right]\\
&=-AB\alpha^{s(m+1)}\beta^{sn}\Theta_{\alpha}(s)\Theta_{\beta}(s)-AB\alpha^{sn}\beta^{s(m+1)}\Theta_{\beta}(s)\Theta_{\alpha}(s)\\
&\ \ + AB\alpha^{sm}\beta^{s(n+1)}\Theta_{\alpha}(s)\Theta_{\beta}(s)+AB\alpha^{s(n+1)}\beta^{sm}\Theta_{\beta}(s)\Theta_{\alpha}(s)\\
&=AB\left(\alpha^{s}-\beta^{s}\right)\left[\alpha^{sn}\beta^{sm}\Theta_{\beta}(s)\Theta_{\alpha}(s)-\alpha^{sm}\beta^{sn}\Theta_{\alpha}(s)\Theta_{\beta}(s)\right]\\
&=AB(\alpha\beta)^{sm}\left(\alpha^{s}-\beta^{s}\right)\left[\alpha^{s(n-m)}\Theta_{\beta}(s)\Theta_{\alpha}(s)-\beta^{s(n-m)}\Theta_{\alpha}(s)\Theta_{\beta}(s)\right]\\
&=AB(\alpha\beta)^{sm}\left(\alpha^{s}-\beta^{s}\right)\left\lbrace \begin{array}{c} 
\left(\alpha^{s(n-m)}-\beta^{s(n-m)}\right)\Theta_{\alpha}(s)\Theta_{\beta}(s)\\
+ \alpha^{s(n-m)}\left[\Theta_{\beta}(s),\Theta_{\alpha}(s)\right]
\end{array}\right\rbrace ,
\end{align*}
where $\left[\Theta_{\beta}(s),\Theta_{\alpha}(s)\right]=2(\alpha\beta)^{s}(\alpha^{s}-\beta^{s})\left[\lambda_{3}(\alpha\beta)^{s}\textbf{i}-\lambda_{2}(\alpha^{s}+\beta^{s})\textbf{j}+\lambda_{1}\textbf{k}\right]$. Hence we get the result.
\end{proof}

Note that for $m=n$ in Theorem \ref{t6} we have the next identity for the higher order Horadam 3-parameter generalized quaternions.
\begin{corollary}
Let $n$ be an integer such that $n\geq 0$. Then
\begin{align*}
\textnormal{\textbf{Q}}W_{n+1}(s)\textnormal{\textbf{Q}}W_{n}(s)&-\textnormal{\textbf{Q}}W_{n}(s)\textnormal{\textbf{Q}}W_{n+1}(s)\\
&=\frac{AB(\alpha\beta)^{sn}\left(\alpha^{s}-\beta^{s}\right)}{\left[A\alpha^{s}-B\beta^{s}\right]^{2}}\left[\Theta_{\beta}(s),\Theta_{\alpha}(s)\right] ,
\end{align*}
where $\left[\Theta_{\beta}(s),\Theta_{\alpha}(s)\right]=2(\alpha\beta)^{s}(\alpha^{s}-\beta^{s})\left[\lambda_{3}(\alpha\beta)^{s}\textnormal{\textbf{i}}-\lambda_{2}(\alpha^{s}+\beta^{s})\textnormal{\textbf{j}}+\lambda_{1}\textnormal{\textbf{k}}\right]$.
\end{corollary}

\section{Other identities of the higher order Horadam 3-parameter generalized quaternions}

\begin{theorem}\label{t7}
Let $n\geq 0$ be an integer. Then the summation formula for the higher order Horadam 3-parameter generalized quaternions is as follows
\[
\sum_{r=0}^{n}\textnormal{\textbf{Q}}W_{r}(s)=\frac{(\alpha\beta)^{s}\textnormal{\textbf{Q}}W_{n}(s)-\textnormal{\textbf{Q}}W_{n+1}(s)+\sigma(s)\Pi}{(\alpha\beta)^{s}-(\alpha^{s}+\beta^{s})+1}-\Phi,
\]
where 
\begin{align*}
\sigma(s)&=W_{1}(s)+\left(1-(\alpha^{s}+\beta^{s})\right)W_{0}(s),\\
\Pi&=\textnormal{\textbf{1}}+\textnormal{\textbf{i}}+\textnormal{\textbf{j}}+\textnormal{\textbf{k}},\\
\Phi&=W_{0}(s)\textnormal{\textbf{i}}+\left[W_{0}(s)+W_{1}(s)\right]\textnormal{\textbf{j}}+\left[(1-(\alpha\beta)^{s})W_{0}(s)+(1+\alpha^{s}+\beta^{s})W_{1}(s)\right]\textnormal{\textbf{k}},
\end{align*}
and $W_{0}(s)$ and $W_{1}(s)$ defined as in (\ref{e1}).
\end{theorem}
\begin{proof}
Using formulas (\ref{f1}) and (\ref{id1}), we get
\begin{align*}
&\sum_{r=0}^{n}\left[W_{r}(s)\textbf{1}+W_{r+1}(s)\textbf{i}+W_{r+2}(s)\textbf{j}+W_{r+3}(s)\textbf{k}\right]\\
&=\textbf{1}\sum_{r=0}^{n}W_{r}(s)+\textbf{i}\sum_{r=1}^{n+1}W_{r}(s)+\textbf{j}\sum_{r=2}^{n+2}W_{r}(s)+\textbf{k}\sum_{r=3}^{n+3}W_{r}(s)\\
&=\left[\frac{(\alpha\beta)^{s}W_{n}(s)-W_{n+1}(s)+\sigma(s)}{(\alpha\beta)^{s}-(\alpha^{s}+\beta^{s})+1}\right]\textbf{1}\\
&\ \ + \left[\frac{(\alpha\beta)^{s}W_{n+1}(s)-W_{n+2}(s)+\sigma(s)}{(\alpha\beta)^{s}-(\alpha^{s}+\beta^{s})+1}-W_{0}(s)\right]\textbf{i}\\
&\ \ + \left[\frac{(\alpha\beta)^{s}W_{n+2}(s)-W_{n+3}(s)+\sigma(s)}{(\alpha\beta)^{s}-(\alpha^{s}+\beta^{s})+1}-W_{0}(s)-W_{1}(s)\right]\textbf{j}\\
&\ \ + \left[\frac{(\alpha\beta)^{s}W_{n+3}(s)-W_{n+4}(s)+\sigma(s)}{(\alpha\beta)^{s}-(\alpha^{s}+\beta^{s})+1}-W_{0}(s)-W_{1}(s)-W_{2}(s)\right]\textbf{k}\\
&=\frac{(\alpha\beta)^{s}\textnormal{\textbf{Q}}W_{n}(s)-\textnormal{\textbf{Q}}W_{n+1}(s)+\sigma(s)\Pi}{(\alpha\beta)^{s}-(\alpha^{s}+\beta^{s})+1}-\Phi,
\end{align*}
where $\sigma(s)=W_{1}(s)+\left(1-(\alpha^{s}+\beta^{s})\right)W_{0}(s)$, $\Pi=\textbf{1}+\textbf{i}+\textbf{j}+\textbf{k}$ and $\Phi=W_{0}(s)\textnormal{\textbf{i}}+\left[W_{0}(s)+W_{1}(s)\right]\textnormal{\textbf{j}}+\left[(1-(\alpha\beta)^{s})W_{0}(s)+(1+\alpha^{s}+\beta^{s})W_{1}(s)\right]\textnormal{\textbf{k}}$.
\end{proof}

\begin{theorem}\label{t8}
Let $n\geq 0$, $m\geq 0$ be integers such that $n\geq m$. Then
\begin{align*}
&\textnormal{\textbf{Q}}W_{m}(s)\textnormal{\textbf{Q}}U_{n}(s)-\textnormal{\textbf{Q}}U_{m}(s)\textnormal{\textbf{Q}}W_{n}(s)\\
&=\frac{(\alpha-\beta)W_{0}(\alpha\beta)^{sm}}{A\alpha^{s}-B\beta^{s}}\left\lbrace \begin{array}{c} 
U_{n-m}(s)\Theta_{\alpha}(s)\Theta_{\beta}(s)\\
+2\alpha^{s(n-m+1)}\beta^{s}\left[\lambda_{3}(\alpha\beta)^{s}\textnormal{\textbf{i}}-\lambda_{2}(\alpha^{s}+\beta^{s})\textnormal{\textbf{j}}+\lambda_{1}\textnormal{\textbf{k}}\right]
\end{array}\right\rbrace .
\end{align*}
\end{theorem}
\begin{proof}
The Binet type formulas for the higher order Horadam and generalized Fibonacci 3-parameter generalized quaternions give
\begin{align*}
&\textnormal{\textbf{Q}}W_{m}(s)\textnormal{\textbf{Q}}U_{n}(s)-\textnormal{\textbf{Q}}U_{m}(s)\textnormal{\textbf{Q}}W_{n}(s)\\
&=\left[\frac{A\alpha^{sm}\Theta_{\alpha}(s)-B\beta^{sm}\Theta_{\beta}(s)}{A\alpha^{s}-B\beta^{s}}\right]\left[\frac{\alpha^{sn}\Theta_{\alpha}(s)-\beta^{sn}\Theta_{\beta}(s)}{\alpha^{s}-\beta^{s}}\right]\\
&\ \ - \left[\frac{\alpha^{sm}\Theta_{\alpha}(s)-\beta^{sm}\Theta_{\beta}(s)}{\alpha^{s}-\beta^{s}}\right]\left[\frac{A\alpha^{sn}\Theta_{\alpha}(s)-B\beta^{sn}\Theta_{\beta}(s)}{A\alpha^{s}-B\beta^{s}}\right]\\
&=\frac{(A-B)\left(\alpha^{sn}\beta^{sm}\Theta_{\beta}(s)\Theta_{\alpha}(s)-\alpha^{sm}\beta^{sn}\Theta_{\alpha}(s)\Theta_{\beta}(s)\right)}{\left[\alpha^{s}-\beta^{s}\right]\left[A\alpha^{s}-B\beta^{s}\right]}\\
&=\frac{(A-B)(\alpha\beta)^{sm}\left(\alpha^{s(n-m)}\Theta_{\beta}(s)\Theta_{\alpha}(s)-\beta^{s(n-m)}\Theta_{\alpha}(s)\Theta_{\beta}(s)\right)}{\left[\alpha^{s}-\beta^{s}\right]\left[A\alpha^{s}-B\beta^{s}\right]}.
\end{align*}
By simple calculations, using (\ref{c}), we have
\begin{align*}
&\textnormal{\textbf{Q}}W_{m}(s)\textnormal{\textbf{Q}}U_{n}(s)-\textnormal{\textbf{Q}}U_{m}(s)\textnormal{\textbf{Q}}W_{n}(s)\\
&=\frac{(A-B)(\alpha\beta)^{sm}U_{n-m}(s)\Theta_{\alpha}(s)\Theta_{\beta}(s)}{A\alpha^{s}-B\beta^{s}}\\
&\ \ + \frac{2(A-B)(\alpha\beta)^{s(m+1)}\alpha^{s(n-m)}\left[\lambda_{3}(\alpha\beta)^{s}\textbf{i}-\lambda_{2}(\alpha^{s}+\beta^{s})\textbf{j}+\lambda_{1}\textbf{k}\right]}{A\alpha^{s}-B\beta^{s}}\\
&=\frac{(\alpha-\beta)W_{0}(\alpha\beta)^{sm}}{A\alpha^{s}-B\beta^{s}}\left\lbrace \begin{array}{c} 
U_{n-m}(s)\Theta_{\alpha}(s)\Theta_{\beta}(s)\\
+2\alpha^{s(n-m+1)}\beta^{s}\left[\lambda_{3}(\alpha\beta)^{s}\textbf{i}-\lambda_{2}(\alpha^{s}+\beta^{s})\textbf{j}+\lambda_{1}\textbf{k}\right]
\end{array}\right\rbrace ,
\end{align*}
as desired.
\end{proof}

Replacing $m=n$ in Theorem \ref{t8}, we obtain the following result.
\begin{corollary}
Let $n$ be an integer such that $n\geq 0$. Then
\begin{align*}
&\textnormal{\textbf{Q}}W_{n}(s)\textnormal{\textbf{Q}}U_{n}(s)-\textnormal{\textbf{Q}}U_{n}(s)\textnormal{\textbf{Q}}W_{n}(s)\\
&=\frac{2(\alpha-\beta)W_{0}(\alpha\beta)^{s(n+1)}}{A\alpha^{s}-B\beta^{s}}\left[\lambda_{3}(\alpha\beta)^{s}\textnormal{\textbf{i}}-\lambda_{2}(\alpha^{s}+\beta^{s})\textnormal{\textbf{j}}+\lambda_{1}\textnormal{\textbf{k}}\right].
\end{align*}
\end{corollary}

\section{Generating functions and matrix generators}

In this section, we will give the generating functions for the higher order Horadam 3-parameter generalized quaternions. Similarly like Horadam and higher order Horadam sequence in (\ref{id2}), these sequences can be considered as the coefficients of the power series expansion of the corresponding generating functions. 

\begin{theorem}
The generating function of the higher order Horadam 3-parameter generalized quaternions has the following form
\[
g\left(\textnormal{\textbf{Q}}W_{n}(s);\lambda\right)=\frac{A\Theta_{\alpha}(s)-B\Theta_{\beta}(s)-\left[A\beta^{s}\Theta_{\alpha}(s)-B\alpha^{s}\Theta_{\beta}(s)\right]\lambda}{\left[A\alpha^{s}-B\beta^{s}\right]\left[1-(\alpha^{s}+\beta^{s})\lambda+(\alpha\beta)^{s}\lambda^{2}\right]}.
\]
\end{theorem}
\begin{proof}
Let
\[
g\left(\textnormal{\textbf{Q}}W_{n}(s);\lambda\right)=\textnormal{\textbf{Q}}W_{0}(s)+\textnormal{\textbf{Q}}W_{1}(s)\lambda+\textnormal{\textbf{Q}}W_{2}(s)\lambda^{2}+\cdots 
\]
be the generating function of the higher order Horadam 3-parameter generalized quaternions. Hence
\begin{align*}
(\alpha^{s}+\beta^{s})\lambda g\left(\textnormal{\textbf{Q}}W_{n}(s);\lambda\right)&=(\alpha^{s}+\beta^{s})\textnormal{\textbf{Q}}W_{0}(s)\lambda+(\alpha^{s}+\beta^{s})\textnormal{\textbf{Q}}W_{1}(s)\lambda^{2}+\cdots \\
(\alpha\beta)^{s}\lambda^{2} g\left(\textnormal{\textbf{Q}}W_{n}(s);\lambda\right)&=(\alpha\beta)^{s}\textnormal{\textbf{Q}}W_{0}(s)\lambda^{2}+(\alpha\beta)^{s}\textnormal{\textbf{Q}}W_{1}(s)\lambda^{3}+\cdots .
\end{align*}
Using the recurrence (\ref{e2}), we get
\begin{align*}
\left[1-(\alpha^{s}+\beta^{s})\lambda+(\alpha\beta)^{s}\lambda^{2}\right]&g\left(\textnormal{\textbf{Q}}W_{n}(s);\lambda\right)\\
&=\textnormal{\textbf{Q}}W_{0}(s)+\left[\textnormal{\textbf{Q}}W_{1}(s)-(a^{s}+\beta^{s})\textnormal{\textbf{Q}}W_{0}(s)\right]\lambda .
\end{align*}
Thus
\[
g\left(\textnormal{\textbf{Q}}W_{n}(s);\lambda\right)=\frac{\textnormal{\textbf{Q}}W_{0}(s)+\left[\textnormal{\textbf{Q}}W_{1}(s)-(a^{s}+\beta^{s})\textnormal{\textbf{Q}}W_{0}(s)\right]\lambda}{1-(\alpha^{s}+\beta^{s})\lambda+(\alpha\beta)^{s}\lambda^{2}}.
\]
Using Theorem \ref{t3}, we obtain
\begin{align*}
\textnormal{\textbf{Q}}W_{0}(s)&=\frac{A\Theta_{\alpha}(s)-B\Theta_{\beta}(s)}{A\alpha^{s}-B\beta^{s}},\\
\textnormal{\textbf{Q}}W_{1}(s)&=\frac{A\alpha^{s}\Theta_{\alpha}(s)-B\beta^{s}\Theta_{\beta}(s)}{A\alpha^{s}-B\beta^{s}}.
\end{align*}
Finally, we can write 
\[
g\left(\textnormal{\textbf{Q}}W_{n}(s);\lambda\right)=\frac{A\Theta_{\alpha}(s)-B\Theta_{\beta}(s)-\left[A\beta^{s}\Theta_{\alpha}(s)-B\alpha^{s}\Theta_{\beta}(s)\right]\lambda}{\left[A\alpha^{s}-B\beta^{s}\right]\left[1-(\alpha^{s}+\beta^{s})\lambda+(\alpha\beta)^{s}\lambda^{2}\right]}.
\]
\end{proof}

In the same way we can prove the next results.
\begin{corollary}
The generating function of the higher order generalized Fibonacci 3-parameter generalized quaternions has the following form
\[
g\left(\textnormal{\textbf{Q}}U_{n}(s);\lambda\right)=\frac{\Theta_{\alpha}(s)-\Theta_{\beta}(s)-\left[\beta^{s}\Theta_{\alpha}(s)-\alpha^{s}\Theta_{\beta}(s)\right]\lambda}{\left[\alpha^{s}-\beta^{s}\right]\left[1-(\alpha^{s}+\beta^{s})\lambda+(\alpha\beta)^{s}\lambda^{2}\right]}.
\]
\end{corollary}

\begin{theorem}
The exponential generating function of the higher order Horadam 3-parameter generalized quaternions is
\[
\sum_{n\geq 0}\textnormal{\textbf{Q}}W_{n}(s)\frac{\lambda^{n}}{n!}=\frac{A\Theta_{\alpha}(s)e^{\alpha^{s}\lambda}-B\Theta_{\beta}(s)e^{\beta^{s}\lambda}}{A\alpha^{s}-B\beta^{s}}.
\]
\end{theorem}
\begin{proof}
Using Theorem \ref{t3}, we have
\begin{align*}
\sum_{n\geq 0}\textnormal{\textbf{Q}}W_{n}(s)\frac{\lambda^{n}}{n!}&=\sum_{n\geq 0}\left[\frac{A\alpha^{sn}\Theta_{\alpha}(s)-B\beta^{sn}\Theta_{\beta}(s)}{A\alpha^{s}-B\beta^{s}}\right]\frac{\lambda^{n}}{n!}\\
&=\frac{A\Theta_{\alpha}(s)\sum_{n\geq 0}\alpha^{sn}\frac{\lambda^{n}}{n!}-B\Theta_{\beta}(s)\sum_{n\geq 0}\beta^{sn}\frac{\lambda^{n}}{n!}}{A\alpha^{s}-B\beta^{s}}\\
&=\frac{A\Theta_{\alpha}(s)e^{\alpha^{s}\lambda}-B\Theta_{\beta}(s)e^{\beta^{s}\lambda}}{A\alpha^{s}-B\beta^{s}}.
\end{align*}
\end{proof}

Now, we introduce a matrix generator for the higher order Horadam numbers. The higher order Horadam $\textrm{H}(s)$-matrix is defined by
\begin{equation}\label{mat}
\textrm{H}(s)=\left[\begin{array}{rr} \alpha^{s}+\beta^{s}& -(\alpha\beta)^{s}\\
1&0\end{array}\right].
\end{equation}

The following result holds for the higher order Horadam $\textrm{H}(s)$-matrix.
\begin{theorem}\label{t9}
Let $\textnormal{\textrm{H}}(s)$ be the higher order Horadam matrix given in (\ref{mat}). Then for every positive integer $n$, we have
\[
\left[\textnormal{\textrm{H}}(s)\right]^{n}=\left[\begin{array}{rr} U_{n+1}(s)& -(\alpha\beta)^{s}U_{n}(s)\\
U_{n}(s)&-(\alpha\beta)^{s}U_{n-1}(s)\end{array}\right].
\]
\end{theorem}
\begin{proof}
If $n=1$ then the result is obvious. Assuming the result holds for $n$, we will prove it for $n+1$. By the induction's hypothesis and relation (\ref{e3}) we get
\begin{align*}
\left[\textnormal{\textrm{H}}(s)\right]^{n+1}&=\left[\textnormal{\textrm{H}}(s)\right]^{n}\textnormal{\textrm{H}}(s)\\
&=\left[\begin{array}{rr} U_{n+1}(s)& -(\alpha\beta)^{s}U_{n}(s)\\
U_{n}(s)&-(\alpha\beta)^{s}U_{n-1}(s)\end{array}\right]\left[\begin{array}{rr} \alpha^{s}+\beta^{s}& -(\alpha\beta)^{s}\\
1&0\end{array}\right]\\
&=\left[\begin{array}{rr} (\alpha^{s}+\beta^{s})U_{n+1}(s)-(\alpha\beta)^{s}U_{n}(s)& -(\alpha\beta)^{s}U_{n+1}(s)\\
(\alpha^{s}+\beta^{s})U_{n}(s)-(\alpha\beta)^{s}U_{n-1}(s)&-(\alpha\beta)^{s}U_{n}(s)\end{array}\right]\\
&=\left[\begin{array}{rr} U_{n+2}(s)& -(\alpha\beta)^{s}U_{n+1}(s)\\
U_{n+1}(s)&-(\alpha\beta)^{s}U_{n}(s)\end{array}\right].
\end{align*}
Thus, the proof is completed.
\end{proof}

In the same way, using Theorem \ref{t1}, one can easily prove the next result.
\begin{theorem}\label{t10}
Let $n\geq 1$ be an integer. Then
\[
\left[\begin{array}{rr} \textnormal{\textbf{Q}}W_{n+1}(s)& -(\alpha\beta)^{s}\textnormal{\textbf{Q}}W_{n}(s)\\
\textnormal{\textbf{Q}}W_{n}(s)&-(\alpha\beta)^{s}\textnormal{\textbf{Q}}W_{n-1}(s)\end{array}\right]=\left[\begin{array}{rr} \textnormal{\textbf{Q}}W_{2}(s)& -(\alpha\beta)^{s}\textnormal{\textbf{Q}}W_{1}(s)\\
\textnormal{\textbf{Q}}W_{1}(s)&-(\alpha\beta)^{s}\textnormal{\textbf{Q}}W_{0}(s)\end{array}\right]\left[\textnormal{\textrm{H}}(s)\right]^{n-1}.
\]
\end{theorem}

We now present some results that can be derived from Theorem \ref{t10}.
\begin{corollary}
Let $n\geq 0$, $m\geq 1$ be integers. Then
\[
\textnormal{\textbf{Q}}W_{n+m+1}(s)=\textnormal{\textbf{Q}}W_{2}(s)U_{n+m}(s)-(\alpha\beta)^{s}\textnormal{\textbf{Q}}W_{1}(s)U_{n+m-1}(s).
\]
\end{corollary}

\begin{corollary}
Let $n\geq 1$ be an integer. Then
\[
\textnormal{\textbf{Q}}W_{2n}(s)=\textnormal{\textbf{Q}}W_{1}(s)U_{2n}(s)-(\alpha\beta)^{s}\textnormal{\textbf{Q}}W_{0}(s)U_{2n-1}(s).
\]
\end{corollary}

\section{Conclusion}
In the present paper, firstly, the higher order Horadam 3-parameter generalized quaternions are defined, and several properties involving these 3-parameter generalized quaternions are studied. These newly defined higher order Horadam numbers are the generalized form of the previously introduced higher order Fibonacci numbers given in \cite{Ki1,Ki2,Ki3}. After that, as an extension of the higher order Horadam quaternions, the higher order Horadam generalized quaternions are defined, and many formulas and identities for these quaternions are obtained. The higher order Horadam 3-parameter generalized quaternions are the generalized form of the previously introduced higher order Horadam numbers given in \cite{Kul}.

\end{document}